\documentclass[review,1p]{elsarticle}

\usepackage{amsthm}
\usepackage{amsmath}
\usepackage{amssymb}

\newtheorem{theorem}{Theorem}
\newtheorem{lemma}{Lemma}
\newtheorem{remark}{Remark}
\newtheorem{cons}{Corollary}
\newtheorem{proposition}{Proposition}
\newtheorem{definition}{Definition}

\newcommand{\rank}{\mathrm{rank}\kern 2pt}

\renewcommand{\epsilon}{\varepsilon}

\renewcommand{\geq}{\geqslant}

\begin{document}

\title{Probabilistic estimation of the $\rank 1$ cross approximation accuracy}

\author[inm,mipt]{Osinsky~A.I.}
\ead{sasha\_o@list.ru}

\address[inm]{Institute of Numerical Mathematics RAS, Moscow, Gubkina str., 8, Russia}
\address[mipt]{Moscow Institute of Physics and Technology, Dolgoprudny, Institutsky per., 9, Russia}

\begin{abstract} 

In the construction of low-rank matrix approximation and maximum element search it is effective to use $\operatorname{maxvol}$ algorithm \cite{ALL-2010}. Nevertheless, even in the case of rank 1 approximation the algorithm does not always converge to the maximum matrix element, and it is unclear how often close to the maximum element can be found. In this article it is shown that with a certain degree of randomness in the matrix and proper selection of the starting column, the algorithm with high probability in a few steps converges to an element, which module differs little from the maximum. It is also shown that with more severe restrictions on the error matrix no restrictions on the starting column need to be introduced.

\noindent{\it AMS classification:} 65F30, 65F99, 65D05
\end{abstract}

\begin{keyword}
Low rank approximations; Pseudoskeleton approximations; Maximum volume principle  
\end{keyword}

\maketitle

\section{Introduction}
Among the methods of low-rank matrix approximation, an important place is held by the so-called cross or skeleton approximation. 
In this method approximation of a matrix $A$ is constructed as a product $C \hat A^{-1} R$, where $\hat A \in \mathbb{C}^{r \times r}$ is a submatrix of $A$, located at the intersection of columns $C$ and rows $R$ of $A$.

The resulting low-rank decomposition can be used for approximation of the original matrix, and for searching the maximum element. The similar approach is used for  tensor approximations \cite{OIV-TEE-2010}.

The accuracy of the skeleton and associated pseudoskeleton $CGR$ decompositions is guaranteed in case of submatrix $\hat A$ close to maximum volume \cite{OIV-TEE-2010, GSA-TEE-2001} or, more generally, maximum projective volume \cite{ZNL-OAI-2016}.
However, in general case, the search of the maximum volume submatrix is a NP-hard problem, and these estimates are not directly applicable.

One of the most popular methods for constructing cross low-rank approximation is the algorithm $\operatorname{maxvol}$ \cite{ALL-2010}. 
In the particular case of rank $1$ approximation it finds the maximum in modulus element in a randomly chosen column, then in the corresponding row (with the maximum element), and so on. 
Finally the resulting element is maximal in modulus element of its row and column. 
Unfortunately, this does not guarantee that it is maximal in the whole matrix (or even close to it). Therefore we can not guarantee that the obtained 
approximation will be accurate enough.

However, in practice, the obtained by  $\operatorname{maxvol} $ algorithm cross approximation is often a good approximation of the original matrix.
This probably means that if the matrix elements are in some sense random, the element found with the help of $ \operatorname{maxvol} $ is likely to be close to the maximum.

Estimates even for the rank 1 particular case are very important, since, for example, we can construct an approximation of rank $ k $ by applying an algorithm $ k $ times.

In Section~\ref{prob-sec} some probability estimates are obtained. 
In Section~\ref{th-sec} we prove the theorems that guarantee a high probability of obtaining sufficiently accurate approximation of rank 1. 
Finally, in Section~\ref{calc-sec} the results of numerical experiments with random matrices are shown and analized.

\section{Probability estimates for some important distributions}\label{prob-sec}

First of all, we need some general propositions about random variables.

\begin{proposition}\label{pr}
Let the random variable $x$ have distribution $\chi^2$ with $n > 2$ degrees of freedom. 
Then for a constant $c$ (the relevant values of $c$ will be determined later) the following holds
\[
\begin{gathered}
\mathbb{P}(x > n - 2 + 2\sqrt{c(n - 2) \ln n}) \leqslant \alpha n^{-c}, \hfill \\
\alpha = \left( {\frac{1}{{\sqrt {\pi (n - 2)}  }} + \frac{1}{2{\sqrt {c\pi \ln n} }}} \right){e^{\frac{{4\sqrt {\frac{{{c^3}{{\ln }^3}n}}{{n - 2}}} }}{3}}}.
\end{gathered}
\]
\end{proposition}
\begin{proof}
We use the expression for the density distribution $ \chi^2 $ and integrate, evaluating the Gamma function from below using Stirling's formula ($n! \geqslant \sqrt{2 \pi n} \left( \frac{n}{e} \right)^n$):
\begin{align}
  \mathbb{P}(x > n - 2 + 2\sqrt{c(n - 2) \ln n}) & = \int\limits_{n - 2 + 2\sqrt {c(n - 2)\ln n} }^\infty  {\frac{{{x^{\frac{n}{2} - 1}}{e^{ - \frac{x}{2}}}}}{{{2^{\frac{n}{2}}}\Gamma \left( {\frac{n}{2}} \right)}}dx} \hfill \nonumber \\
  & \leqslant \int\limits_{n - 2 + 2\sqrt {c(n - 2)\ln n} }^\infty  {\frac{{{x^{\frac{n}{2} - 1}}{e^{ - \frac{x}{2}}}}}{{{2^{\frac{n}{2}}}\sqrt {\pi (n - 2)} {{\left( {\frac{{n - 2}}{{2e}}} \right)}^{\frac{n}{2} - 1}}}}dx}  \hfill \nonumber \\
   & = /y = x - n + 2/ \hfill \nonumber \\ 
   & = \int\limits_{2\sqrt {c(n - 2)\ln n} }^\infty  {\frac{{{{\left( {n - 2 + y} \right)}^{\frac{n}{2} - 1}}{e^{ - \frac{{n - 2}}{2} - \frac{y}{2}}}}}{{2\sqrt {\pi (n - 2)} {{\left( {\frac{{n - 2}}{e}} \right)}^{\frac{n}{2} - 1}}}}dy}  \hfill \nonumber \\
   & = \int\limits_{2\sqrt {c(n - 2)\ln n} }^\infty  {\frac{{{{\left( {1 + \frac{y}{{n - 2}}} \right)}^{\frac{{n - 2}}{2}}}{e^{ - \frac{y}{2}}}}}{{2\sqrt {\pi (n - 2)} }}dy} \hfill \nonumber \\
   & = /z = y - (n - 2)\ln \left( {1 + \frac{y}{{n - 2}}} \right), dy = \frac{{1 + \frac{y}{{n - 2}}}}{{\frac{y}{{n - 2}}}}dz/ \hfill \nonumber \\
   & = \int\limits_{2\sqrt {c(n - 2)\ln n}  - (n - 2)\ln \left( {1 + 2\sqrt {\frac{{c\ln n}}{{n - 2}}} } \right)}^\infty  {\frac{{\left( {1 + \frac{y}{{n - 2}}} \right){e^{ - \frac{z}{2}}}}}{{2\frac{y}{{n - 2}}\sqrt {\pi (n - 2)} }}dz} \hfill \nonumber \\
   & \leqslant \frac{{1 + 2\sqrt {\frac{{c\ln n}}{{n - 2}}} }}{{2\sqrt {\frac{{c\ln n}}{{n - 2}}} \sqrt {\pi (n - 2)} }}{e^{ - \frac{{2\sqrt {c(n - 2)\ln n}  - (n - 2)\ln \left( {1 + 2\sqrt {\frac{{c\ln n}}{{n - 2}}} } \right)}}{2}}} \hfill \nonumber \\
   & \leqslant \frac{{1 + 2\sqrt {\frac{{c\ln n}}{{n - 2}}} }}{2{\sqrt {c\pi \ln n} }}{e^{ - c\ln n + \frac{4}{3}\sqrt {\frac{{{c^3}{{\ln }^3}n}}{{n - 2}}} }} \hfill \nonumber \\
   & = \left( {\frac{1}{{\sqrt {\pi (n - 2)}  }} + \frac{1}{2{\sqrt {c\pi \ln n} }}} \right){e^{\frac{4}{3}\sqrt {\frac{{{c^3}{{\ln }^3}n}}{{n - 2}}} }}{n^{ - c}}. \hfill \nonumber
\end{align}
\end{proof}

Proposition \ref{pr} is applicable only for sufficiently small $c$. Indeed,
\[
{e^{\frac{4}{3}\sqrt {\frac{{{c^3}{{\ln }^3}n}}{{n - 2}}} }}{n^{ - c}} = {n^{\frac{4}{3}c\sqrt {\frac{{c\ln n}}{{n - 2}}} }}{n^{ - c}} = {n^{ - c\left( {1 - \frac{4}{3}\sqrt {\frac{{c\ln n}}{{n - 2}}} } \right)}}.
\]
Thus, the probability is of the order $ n ^ {- c} $ only if $ \sqrt {\frac {c \ln n} {n - 2}} \ll 1 $. However, the condition on $c$ is pretty weak, so we further suppose $c$ to be large enough (e.g. $c > 1$).

Using Proposition~\ref{pr} we can prove the following Lemma. 

\begin{lemma}\label{l1}
Let the random vector $v$ be uniformly distributed on the sphere in the space $\mathbb{C}^{n}$. 
Then, with probability $1 - \alpha n^{-c} - \beta^k$ there is at least one  among any $k$ preselected elements that is not less in absolute value than $\tau$, with
\begin{eqnarray}
\alpha & = & \left( {\frac{1}{{\sqrt {\pi (n - 2)}  }} + \frac{1}{2{\sqrt {c\pi \ln n} }}} \right){e^{\frac{4}{3}\sqrt {\frac{{{c^3}{{\ln }^3}n}}{{n - 2}}} }}, \nonumber \\
\beta & = & \sqrt{ \frac{2\tau^2 \left( n - 2 + 2\sqrt{c(n - 2) \ln n} \right)}{\pi}}. \nonumber
\end{eqnarray}
\end{lemma}
\begin{proof}
Such a vector can be obtained by taking normally distributed random variables, choosing a random rotation of each component in $ \mathbb {C} $ and normalizing. Thus, if as a basis we take the values $ x_i $, then $ | x_i | ^ 2 \sim \chi ^ 2 (1) $, and
\[
|v_i|^2 = \frac{|x_i|^2}{\sum\limits_{j = 1}^n {|x_j|^2}}.
\]
From proposition \ref{pr} we see that
\[
\mathbb{P}(\sum\limits_{i = 1}^n {|x_i|^2} > n - 2 + 2\sqrt{c(n - 2) \ln n}) \leqslant \alpha n^{-c},
\]
Besides,
\begin{eqnarray}
\mathbb{P}\left( |x_i|^2 < t^2 \right) = \mathbb{P}\left(|x_i| < t \right) & = & 2\int\limits_0^{ t } {\frac{{{e^{ - \frac{{{x^2}}}{2}}}}}{{\sqrt {2\pi } }}dx}  \leqslant \sqrt {\frac{2t^2}{\pi}}. \nonumber \\
\mathbb{P}\left( |x_i|^2 < t^2, i = \overline {1,k} \right) & \leqslant & \left( \frac{2t^2}{\pi} \right)^{\frac{k}{2}}. \nonumber
\end{eqnarray}
Eventually, by choosing $t = \tau \sqrt{n - 2 + 2\sqrt{c(n - 2) \ln n}}$, we find that
\[
\mathbb{P}\left(|v_i| < \tau, \; i = \overline{1,k}\right) \leqslant \alpha n^{-c} + \left( \frac{2\tau^2 \left( n - 2 + 2\sqrt{c(n - 2) \ln n} \right)}{\pi} \right)^{\frac{k}{2}}.
\]
\end{proof}

\section{Theorems on the probability of receiving a good low-rank approximations}\label{th-sec}

\begin{theorem}\label{thw}
Let $A = \sigma u v^* + E$, $\sigma > 0$, $\left\| u \right\|_2 = \left\| v \right\|_2 = 1$, $A \in \mathbb{C}^{m \times n}$.
Let vector $v$ be uniformly distributed on the sphere in $\mathbb{C}^n$, $n > 2$.
Let us denote
\[
  \delta = \| E \|_C.
\]
Let
\begin{equation}\label{ed}
\varepsilon  = \frac{{{{\left\| E \right\|}_C}}}{{{{\left\| {A - E} \right\|}_C}}} = \frac{{{{\left\| E \right\|}_C}}}{{\sigma {{\left\| u \right\|}_\infty }{{\left\| v \right\|}_\infty }}} \leqslant \frac{1}{8}.
\end{equation}
Let
\begin{eqnarray}
\alpha & = & \left( {\frac{1}{{\sqrt {\pi (n - 2)}  }} + \frac{1}{2{\sqrt {c\pi \ln n} }}} \right){e^{\frac{4}{3}\sqrt {\frac{{{c^3}{{\ln }^3}n}}{{n - 2}}} }}, \nonumber \\
\beta_v & = & \frac{\left( 1 - \sqrt{1 - 8 \varepsilon}\right) \left\| v \right\|_{\infty} \sqrt{ n - 2 + 2\sqrt{c(n - 2) \ln n} } }{\sqrt{2\pi}}. \nonumber
\end{eqnarray}
Let the algorithm $\operatorname{maxvol}$ \cite{ALL-2010}, on the first step of which we choose a maximal element among the first $k$ columns of the matrix return element $a$ on the intersection of the row $r$ and column $c$. Then with probability $1 - \alpha n^{-c} - \beta_v ^ k$
\begin{equation}\label{th}
\left\| A - c a^{-1} r \right\|_C \leqslant 8 \delta \frac{1 + \varepsilon}{1+\sqrt{1-8\varepsilon}-2\varepsilon}.
\end{equation}
\end{theorem}
\begin{proof} 
Consider an arbitrary element of the matrix $A$:
\[
a_{ij} = \sigma u_i v_j + e_{ij}.
\] 
Fix the corresponding $j$-th column. 
Let 
\[  \mu = \frac{|v_j|}{\left\| v \right\|_{\infty}}.\]

Consider the maximum in modulus element $a_{sj}$ in this column.
It is easy to see that 
\[
|a_{sj}| \geqslant \sigma \mu \left\| u \right\|_{\infty} \left\| v \right\|_{\infty} - \delta.
\]
(This estimate can be obtained by taking into account the row $s_0$ with $|u_{s_0}| = \|u\|_\infty),$ and $|a_{sj}|  \geqslant |a_{s_0 j}|).$ 
 
We will find conditions on $\mu$, which guarantee that the following inequality holds
\[
|u_s| > \mu \left\| u \right\|_{\infty}.
\]
If this is not true (if $|u_s| \leqslant \mu \|u\|_\infty$), we can get the inequality
\[
|a_{sj}| \leqslant
\sigma \mu^2 \left\| u \right\|_{\infty} \left\| v \right\|_{\infty} + \delta.
\]

From the above two estimations for $|a_{sj}| $  we get the condition
\begin{eqnarray}
\sigma \mu^2 \left\| u \right\|_{\infty} \left\| v \right\|_{\infty} + \delta & < & \sigma \mu \left\| u \right\|_{\infty} \left\| v \right\|_{\infty} - \delta \nonumber \\
\mu^2 - \mu + 2\varepsilon & < & 0. \nonumber
\end{eqnarray}
Solving the quadratic equation, we obtain the following condition on $\mu$:
\[
\mu_1 = \frac{1 - \sqrt{1 - 8\varepsilon}}{2} < \mu < \frac{1 + \sqrt{1 - 8\varepsilon}}{2} = \mu_2.
\]

If $\mu \geqslant \mu_2$, then $|u_s | \geqslant \mu_2 \left\| u \right\|_{\infty}$.
Indeed, in this case it is necessary to verify the inequality
\begin{eqnarray}
\sigma \mu \mu_2 \left\| u \right\|_{\infty} \left\| v \right\|_{\infty} + \delta & \leqslant & \sigma \mu \left\| u \right\|_{\infty} \left\| v \right\|_{\infty} - \delta, \nonumber \\
\mu \mu_2 + \varepsilon & \leqslant & \mu - \varepsilon \nonumber
\end{eqnarray}
Noting that $\mu_2 + \mu_1 = 1$, we get
\[
\mu \mu_1 \geqslant 2\varepsilon.
\]
Since
\[
\mu_1 \mu_2 = 2\varepsilon,
\]
the resulting inequality is equivalent to the following:
\[
1 \leqslant \frac{\mu}{\mu_2},
\]
which is true when $\mu \geqslant \mu_2$.
We will get the same conditions, if we swap rows and columns.

These estimates allow us to understand the conditions of halting the algorithm $\operatorname{maxvol}.$
Indeed, let $a_{ij}$ be an element of the matrix $A$, which is the maximum in the $i$-th row and 
$j$-th column of $A$ 
(the element on which the algorithm $\operatorname{maxvol}$ stops). 

Denote $\mu_u = \frac{|u_i|}{\|u\|_\infty}$ and $\mu_v = \frac{|v_j|}{\|v\|_\infty}$. 
We prove that $\mu_u$ and $\mu_v$ at the same time satisfy one of two conditions:
they both are either not greater than $\mu_1$
\[
\mu_u \leqslant \mu_1, \;\;\;\; \mu_v \leqslant \mu_1,
\]
or not less than $\mu_2$ 
\[
\mu_u \geqslant \mu_2, \;\;\;\; \mu_v \geqslant \mu_2.
\]
First, suppose, for example, $\mu_1 < \mu_v < \mu_2$. 
Then, as proved earlier for the element $u_i$ corresponding to the maximum in modulus element of the column, the following inequality is satisfied
\[
|u_i| = \mu_u \|u\|_\infty > \mu_v \|u\|_\infty.
\]
Since $a_{ij}$ is also maximal in the row, by repeating the reasoning, we come to the contradiction
\[
|v_j| = \mu_v \|v\|_\infty > \mu_u \|v\|_\infty > \mu_v \|v\|_\infty.
\]
Thus, neither $\mu_u$, nor $\mu_v$ can be inside the interval $\left( \mu_1,\mu_2 \right).$

It remains to prove the impossibility of the fact that $\mu_u$ and $\mu_v$ are separated by the interval $\left( \mu_1, \mu_2\right).$ Assume, for example, that $\mu_u \leqslant \mu_1$ and $\mu_v \geqslant \mu_2.$ Since $a_{ij}$ is the maximum in $j$-th column, and $\mu_v \geqslant \mu_2$, then, as proved earlier, $\mu_u \geqslant \mu_2$, which contradicts the assumption.

From this we conclude that if at the first step we got to the element with $|v_j| > \mu_1 \left\| v \right\|_{\infty}$, then the value of $\mu$ will increase and eventually will not be less than $\mu_2$.

Let's call the columns (rows) with $|v_j| > \mu_1 \left\| v \right\|_{\infty}$ 
($|u_i| \geqslant \mu_1 \left\| u \right\|_{\infty}$) ``good'', and the others ``bad''. 

By  Lemma~\ref{l1} with $\tau = \mu_1 \left\| v \right\|_{\infty}$ there is at least one ``good'' column among the first $k$ columns with high probability. 
We will show that in this case the maximum in modulus element among these $k$ columns needs to belong to a ``good'' column.

In any ``good'' column (with $|v_{j_0}| > \mu_1 \left\| v \right\|_{\infty}$) there is an element corresponding to $|u_{i_0}| = \|u\|_{\infty}$. Thus
\[
|a_{i_0 j_0}| > \sigma \mu_1 \left\| u \right\|_{\infty} \left\| v \right\|_{\infty} - \delta, 
\]
Then for the maximum in modulus element $a_{ij}$ among these $k$ columns the inequality holds even more so: 
\begin{equation}
\label{eq:max-el-1}
|a_{ij}| \geqslant \sigma \mu_1 \left\| u \right\|_{\infty} \left\| v \right\|_{\infty} - \delta, 
\end{equation}

From the equation for $\mu_1$,
\[
\sigma \mu_1 \left\| u \right\|_{\infty} \left\| v \right\|_{\infty} - \delta = 
\sigma \mu_1^2 \left\| u \right\|_{\infty} \left\| v \right\|_{\infty} + \delta,
\]
we substitute the right-hand side in (\ref{eq:max-el-1}) to get
\begin{equation}
\label{eq:max-el-2}
|a_{ij}| > \sigma \mu_1^2 \left\| u \right\|_{\infty} \left\| v \right\|_{\infty} + \delta. 
\end{equation}
A consequence of (\ref{eq:max-el-2}) is the inequality on the product of $\mu_u$ and $\mu_v$
$$
\mu_u \cdot \mu_v > \mu_1^2.
$$
Thus, if $a_{ij}$ does not belong to the ``good" column ($\mu_v \leqslant \mu_1$), then $\mu_u > \mu_1$ and $a_{ij}$ belongs to the ``good" row.

However, in this case due to the fact that $a_{ij}$ is the maximum in modulus in its column, then, as proved earlier, either $\mu_v \geqslant \mu_2 > \mu_1$, or $\mu_v > \mu_u > \mu_1$,
and, on the contrary to the initial assumption, the column containing $a_{ij}$ is ``good". 

Thus, as a result of the procedure $\operatorname{maxvol}$, we get the element with modulus not less than
\[
\sigma \mu_2 \left\| u \right\|_{\infty} \left\| v \right\|_{\infty} - \delta = \sigma \mu_2^2 \left\| u \right\|_{\infty} \left\| v \right\|_{\infty} + \delta.
\]
Consider a submatrix $2 \times 2$ of $A$:
\[
\hat A = \left[ {\begin{array}{*{20}{c}}
  a&b \\ 
  c&d 
\end{array}} \right],
\]
where $a$ is the element found with $\operatorname{maxvol}$. For the absolute values of the elements $a$ and $d$ following estimates hold
\begin{eqnarray*}
|d| & \leqslant & \sigma \|u\|_\infty \|v\|_\infty + \delta = \sigma \|u\|_\infty \|v\|_\infty \left(1 + \varepsilon \right), \\ 
|a| & \geqslant & \sigma \mu_2 \|u\|_\infty \|v\|_\infty - \delta = \sigma \|u\|_\infty \|v\|_\infty \left(\mu_2 - \varepsilon \right). 
\end{eqnarray*}
Using Theorem 1 from \cite{OIV-TEE-2010}, we get that even if $|d| > |a|$
\[
\left |d - ba^{-1}c \right | = \left | a - bd^{-1}c\right | \frac{|d|}{|a|} \leqslant 4\delta \frac{1 + \varepsilon}{\mu_2 - \varepsilon}.
\]
Substituting the expression for $\mu_2$ and taking into account that submatrix is arbitrary, we obtain the estimate (\ref{th}).
\end{proof}

\begin{remark}
The theorem remains true with probability $1 - \gamma^k$, if the vector $v$ has no more than $\gamma n$ elements, which differ from the maximum more than $\mu_1$ times. 
This allows us to use the result for different distributions of $v$.
\end{remark}

\begin{cons}
Under the conditions of Theorem \ref{thw}:
\begin{enumerate}
\item
\[
\left\| A - c a^{-1} r \right\|_C \leqslant 4 \delta (1 + 16\varepsilon) \leqslant 12 \delta.
\]
\item
\[
\beta_v \leqslant \frac{8\varepsilon \left\| v \right\|_{\infty} \sqrt{ n - 2 + 2\sqrt{c(n - 2) \ln n} } }{\sqrt{2\pi}}.
\]
\item
In order to make the error satisfy \eqref{th} with the probability not exceeding $(\alpha + 1)n^{-c}$, it is sufficient to take
\[
k = \frac{c \ln n}{\ln \frac{1}{\beta_v}}.
\]
\end{enumerate}
\end{cons}

\begin{cons}
If in Theorem \ref{thw} the matrix is real, then
\[
\left\| A - c a^{-1} r \right\|_C \leqslant 4 \delta (1 + 4\varepsilon) \leqslant 6 \delta.
\]
\end{cons}
\begin{proof}
Changes to the proof can be made when considering the submatrix
\[
\hat A = \left[ {\begin{array}{*{20}{c}}
  a&b \\ 
  c&d 
\end{array}} \right].
\]
The result can be more than $4\delta$ only if $|d| > |a|$, but in this case, as $\mu_2 > \varepsilon$, the matrix $E$ does not affect the sign of $a$ and, more so, the signs on $b$, $c$ and $d$. Therefore $sign(d) = sign(ba^{-1}c)$.
Finally,
\begin{align}
|d - b a^{-1} c| & \leqslant |d| - |b a^{-1} c| \nonumber \\ & \leqslant \sigma \|u\|_\infty \|v\|_\infty (1 + \varepsilon) - \sigma \|u\|_\infty \|v\|_\infty (\mu_2 - \varepsilon) \nonumber \\
& \leqslant 2 \delta + \sigma \|u\|_\infty \|v\|_\infty (1-\mu_2) \nonumber \\
& \leqslant 4 \delta (1 + 4\varepsilon) \leqslant 6 \delta. \nonumber
\end{align}
\end{proof}

\begin{theorem}\label{thw2}
Let under the conditions of Theorem \ref{thw}
\[
\beta_v = \frac{8\varepsilon \left\| v \right\|_{\infty} \sqrt{ n - 2 + 2\sqrt{c(n - 2) \ln n} } }{\sqrt{2\pi}}.
\]
Let the algorithm $\operatorname{maxvol}$ perform just $4$ steps, and return an element regardless of whether it is maximal in its column, or not. 
Then
\[
\left\| A - c a^{-1} r \right\|_C \leqslant 4 \delta \left(1 + 16\varepsilon\right).
\]
\end{theorem}
\begin{proof}
It turns out that  with the probability $1-\beta_v^k$ 
\[
|v_j| > 4\varepsilon \left\| v \right\|_{\infty} = \nu_1 \left\| v \right\|_{\infty},
\]
(it is easy to obtain by replacing $\mu_1$ by $\nu_1$).

Let the element found in column (row) be $\nu_k$ of the maximum in $u$ ($v$). Then, if the next found element is 
$\nu_{k+1}$ from maximum in column (row), the following inequality must be satisfied
\[
\nu_k \nu_{k+1} \sigma \left\| u \right\|_{\infty} \left\| v \right\|_{\infty} + \delta \geqslant \nu_k \sigma \left\| u \right\|_{\infty} \left\| v \right\|_{\infty} - \delta.
\]
Therefore
\[
\nu_{k+1} \geqslant 1 - \frac{2 \varepsilon} {\nu_k}.
\]
Substituting $\nu_1 = 4\varepsilon$, we find that

\begin{eqnarray}
\nu_2 & > & \frac{1}{2}, \nonumber \\
\nu_3 & > & 1 - 4\varepsilon. \nonumber
\end{eqnarray}

After the fourth step, both elements will be at least $\nu_3$. 
Indeed, $\mu_2 \geqslant 1 - 4\varepsilon \geqslant 4\varepsilon \geqslant \mu_1$, so, as shown in the proof of Theorem \ref{thw}, when $\mu = 1-4\varepsilon$ is between $\mu_1$ and $\mu_2$, every next coefficient cannot be less. 
Analogously to the Theorem \ref{thw}, we estimate the error of approximation in an arbitrary submatrix of $A$. This will give us the desired estimate for $C$-norm of the error:
\[
|d - ba^{-1}c| \leqslant 4\delta \frac{1+\varepsilon}{\nu_3 - \varepsilon} = 4\delta \frac{1 + \varepsilon}{1 - 5\varepsilon} \leqslant 4\delta (1 + \varepsilon)(1 + \frac{40}{3}\varepsilon) \leqslant 4\delta (1 + 16\varepsilon).
\]
Here we have taken into account that $\varepsilon \leqslant \frac{1}{8}$.
\end{proof}

Thus, in order to find the element, which is close to the maximum, with the prescribed probability, it suffices to compare only $ (k + 1) m + 2n $ elements of $A$.

In all the above estimates $\left\| u \right\|_{\infty} \geqslant \frac{1}{\sqrt m}$, $\left\| v \right\|_{\infty} \geqslant \frac{1}{\sqrt n}$. For upper bounds we can again use a probabilistic approach.

\begin{definition}
Vector $v \in \mathbb{C}^{n}$ is called \textbf{$\mu$-coherent} with the parameter $\mu > 0$, when
\[
\left\| v \right\|_{\infty} \leqslant \sqrt{\frac{\mu}{n}}.
\]
\end{definition}

\begin{proposition}
Let random vector $v$ be uniformly distributed on the sphere in $\mathbb{C}^{n}$, $n > 1$. Then with probability $1 - \frac{n^{-c(1 - \frac{1}{n})}}{\sqrt {c \ln n}}$ it is $\mu$-coherent with the parameter $\mu = 2c \ln n$.
\end{proposition}
\begin{proof}
We construct the vector $v$ as in the Proposition \ref{pr}. Then
\[
\mathbb{P}( |v_i|^2 < t) = \mathbb{P}\left( \frac{|x_i|^2}{\sum\limits_{j = 1}^n {|x_j|^2}} < t \right) = \mathbb{P}\left( \frac{|x_i|^2}{\sum\limits_{\begin{subarray}{l} 
  j = 1 \\ 
  j \ne i 
\end{subarray}}^n {|x_j|^2}} < \frac{t}{1 - t} \right),
\]
\[
\mathbb{P}(|v_i|^2 < t, \; i = \overline{1,n}) \leqslant n \mathbb{P}\left( \frac{|x_1|^2}{\sum\limits_{j = 2}^n {|x_j|^2}} < \frac{t}{1 - t} \right).
\]
Random value $\frac{|x_1|^2}{\sum\limits_{j = 2}^n {|x_j|^2}}$ has Fisher distribution with degrees of freedom $1$ and $n-1$. 
Now we can estimate the probability using the density function:
\[
\begin{gathered}
\mathbb{P}(\left\| v \right\|_{\infty}^2 < t) \leqslant n\int\limits_{(n - 1)\frac{t}{{1 - t}}}^\infty  {\frac{{\sqrt {\frac{{x{{(n - 1)}^{n - 1}}}}{{{{(x + n - 1)}^n}}}} }}{{x{\rm B}\left( {\frac{1}{2},\frac{{n - 1}}{2}} \right)}}dx}  = n\int\limits_{(n - 1)\frac{t}{{1 - t}}}^\infty  {\frac{{\sqrt {\frac{{{{(n - 1)}^{n - 1}}}}{{{{(x + n - 1)}^n}}}} }}{{\sqrt x {\rm B}\left( {\frac{1}{2},\frac{{n - 1}}{2}} \right)}}dx}  \leqslant  \hfill \\
   \leqslant /{x_0} = (n - 1)\frac{t}{{1 - t}}/ \leqslant n\int\limits_{{x_0}}^\infty  {\frac{{\sqrt {{x_0} + n - 1} \sqrt {\frac{{{{(n - 1)}^{n - 1}}}}{{{{(x + n - 1)}^n}}}} }}{{\sqrt {{x_0}(x + n - 1)} {\rm B}\left( {\frac{1}{2},\frac{{n - 1}}{2}} \right)}}} dx =  \hfill \\
   = \frac{{n{{(n - 1)}^{\frac{{n - 1}}{2}}}\sqrt {{x_0} + n - 1} }}{{\sqrt {{x_0}} {\rm B}\left( {\frac{1}{2},\frac{{n - 1}}{2}} \right)}}\int\limits_{{x_0}}^\infty  {\frac{{dx}}{{{{(x + n - 1)}^{\frac{{n + 1}}{2}}}}}}  \leqslant  \hfill \\
   \leqslant \frac{{n{{(n - 1)}^{\frac{{n - 1}}{2}}}\sqrt {{x_0} + n - 1} }}{{\sqrt {{x_0}} }}\frac{{\sqrt {\frac{{n - 1}}{2}} }}{{\sqrt \pi  }}\frac{2}{{n - 1}}{({x_0} + n - 1)^{ - \frac{{n - 1}}{2}}} =  \hfill \\
   = \sqrt {\frac{2}{\pi }} \frac{{n{{(n - 1)}^{\frac{{n - 2}}{2}}}}}{{\sqrt {{x_0}} }}{({x_0} + n - 1)^{ - \frac{{n - 2}}{2}}} =  \hfill \\
   = \sqrt {\frac{2}{\pi }} \frac{{n{{(n - 1)}^{\frac{{n - 2}}{2}}}}}{{\sqrt {(n - 1)\frac{t}{{1 - t}}} }}{\left( {(n - 1)\frac{t}{{1 - t}} + n - 1} \right)^{ - \frac{{n - 2}}{2}}} =  \hfill \\
   = \sqrt {\frac{2}{\pi }} \frac{n}{{\sqrt {(n - 1)t} }}{(1 - t)^{\frac{{n - 1}}{2}}} \leqslant \sqrt {\frac{2}{\pi }} \frac{n}{{\sqrt {(n - 1)t} }}{e^{ - t\frac{{n - 1}}{2}}} \leqslant  \hfill \\
   \leqslant \sqrt {\frac{2}{\pi }} \sqrt {\frac{n}{{\mu (n - 1)}}} {e^{ - \frac{\mu }{2}(1 - \frac{1}{n})}} = \sqrt {\frac{n}{{\pi c(n - 1)\ln n}}} {n^{ - c(1 - \frac{1}{n})}} \leqslant \frac{{{n^{ - c(1 - \frac{1}{n})}}}}{{\sqrt {c \ln n} }}. \hfill \\ 
\end{gathered}
\]
\end{proof}

Condition of $ \mu $-coherence can be used in case of hard-to-estimate $ C $-norm of the matrix $ E $.
Even if we demand its fulfilment for all rows and columns of a random unitary matrix, we can ensure with high probability that $ \mu \sim \ln n $.

\begin{cons}
Let the conditions of Theorem \ref{thw} be fulfilled.
Also let the rows $U \in \mathbb{C}^{m \times m}$ and columns $V \in \mathbb{C}^{n \times n}$  from singular value decomposition $A = U \Sigma V$ be $\mu$-coherent,
and $\sigma = \sigma_1(A)$ be the maximum singular value of the matrix, with the  corresponding singular vectors $u$ and $v$.
Then
\begin{equation}\label{eq-sum}
\delta \leqslant \frac{\mu}{\sqrt{mn}} \sum\limits_{j = 2}^{\min (m,n)} \sigma _j(A).
\end{equation}
If $U$ is a random unitary matrix, then with the probability
\[1 - \frac{nm^{-c(1 - \frac{1}{m})}}{\sqrt {c \ln m}},\]

\begin{equation}\label{eq-2}
\delta \leqslant \sqrt{\frac{2c \ln m}{m}} \sigma_2(A).
\end{equation}
\end{cons}
\begin{proof}
\begin{align}
  {\left\| E \right\|_C} & = \mathop {\max }\limits_{i,k} \left| {\sum\limits_{j = 2}^{\min (m,n)} {{u_{ij}}{\sigma _j}{v_{jk}}} } \right| \nonumber \\
  & \leqslant \mathop {\max }\limits_{i,j} \left| {{u_{ij}}} \right|\mathop {\max }\limits_{j,k} \left| {{v_{jk}}} \right|\left| {\sum\limits_{j = 2}^{\min (m,n)} {{\sigma _j}} } \right| \nonumber \\
  & \leqslant \frac{\mu }{\sqrt{mn}}\sum\limits_{j = 2}^{\min (m,n)} {{\sigma _j}}, \nonumber
\end{align}
which proves the inequality (\ref{eq-sum}).

To prove (\ref{eq-2}) consider the vector $e$ with components $e_k = \sigma_j v_{jk}$, $e_1 = 0$. 
Its Euclidean norm is not greater than $\sigma_2(A)$. By selecting it as one of the basis vectors (with any other orthonormal vectors), we get in the above product simply an element of a random vector $u_i$, but in the new basis. We need to apply the condition on $ mn $ different components, which is equivalent to $n$ uses of $\mu$-coherence. As a result, we obtain the required inequality.
\end{proof}

In addition, from the probability estimate of $\mu$-coherence it is clear that in order to guarantee, that the value of $\beta_v$ is less than 1, it is required that $\varepsilon \sim \frac{1}{\sqrt{\ln n}}$. And, although after entering the ``good'' column, it will require very few steps to get a good estimate, it may be necessary to view a lot of columns to ensure that the column or row will actually be ``good''.

In practice, of course, algorithm is used without viewing the columns. This is, firstly, due to the fact that each step of the algorithm is roughly equivalent to increasing $ k $ by 1. 
In addition, selecting an element corresponding to large value in $ \sigma uv ^ * $ is more probable than selecting the one for a smaller value. 
However, analysis of such probabilities is much more difficult. Nevertheless, it can be done by imposing additional restrictions on the matrix $ E $.

\begin{theorem}
Let under the conditions of Theorem \ref{th} for a matrix $A \in \mathbb{R}^{n \times n}$, vectors $u$ and $v$ are uniformly distributed on the sphere in $\mathbb{R}^n$,
\[
  \beta_u = \frac{\left( 1 - \sqrt{1 - 8 \varepsilon}\right) \left\| u \right\|_{\infty} \sqrt{ n - 2 + 2\sqrt{c(n - 2) \ln n} } }{\sqrt{2\pi}},
\]
and the matrix $E$ consists of independent (including the $ u $ and $ v $) random variables with uniformly distributed on the interval $[0; \delta]$ modules.

Suppose that at the beginning algorithm $\operatorname{maxvol}$ instead of viewing $k$ random columns at least $k$ steps are made, and the maximum element among the viewed ones is selected  (if there are less than $k$ steps, the algorithm continues with the next to the maximum element). 

Then the estimate (\ref{th}) holds with the probability
\[
1 - 2\alpha n^{-c} - \alpha_0 n^{-\gamma k} - \left( \frac{c_0 \ln n}{n} \right)^k.
\]
where
\begin{eqnarray}
  \gamma & = & 1 - \beta  - \frac{2 \varepsilon \|u\|_\infty \|v\|_\infty}{c_0} \cdot \frac{{2\left( {n - 2 + 2\sqrt {c\left( {n - 2} \right)\ln n} } \right)}}{{\pi}}, \nonumber \\
  \alpha & = & \left( {\frac{1}{{\sqrt {\pi (n - 2)}  }} + \frac{1}{2{\sqrt {c\pi \ln n} }}} \right){e^{\frac{4}{3}\sqrt {\frac{{{c^3}{{\ln }^3}n}}{{n - 2}}} }}, \nonumber \\
  \alpha_0 & = & e^{\gamma \frac{k^2 \ln^2 n}{2n}} \nonumber
\end{eqnarray}
with the arbitrary constants $c$ and $c_0$.
\end{theorem}

Thus, each step of the algorithm reduces the probability of error in almost $ n $ times when $ c_0 $ is large enough (about $\mu$ in the case of $\mu$-coherence). 

\begin{proof}
Firstly, we immediately note that due to the independence of the elements of $\sigma u v$ and $E$ the probability to get to the large elements in $\sigma u v^*$ after each step of the algorithm is not lower than simply by browsing a random column or row.
Thus, not less beneficial is just to make $ n $ steps than to seek at the beginning maximum in $ n $ rows or columns.
This ``benefit'' can be evaluated quantitatively.

Taking into account that the elements of the matrices $\sigma u v^*$ and $E$ are independent random variables, consider for each pair of indices $(i, j)$ four conditions
\begin{eqnarray}
\label{eq:cond-1}
|v_j| & > & \mu_0, \\
\label{eq:cond-2}
(\sigma u v^*)_{ij} E_{ij} & > & 0, \\
\label{eq:cond-3}
\sigma |u_i| \mu_0 & \geqslant & \sigma \mu_1 \|u\|_\infty \mu_0 + \varepsilon_0 \delta, \\
\label{eq:cond-4}
|E_{ij}| & \geqslant &  \left(1 - \varepsilon_0 \right) \delta,
\end{eqnarray}
where $\mu_0 > 0$ and $0 < \varepsilon_0 < 1$ -- are some parameters that will be determined later. 

Let us make some observations. Firstly, if the element $a_{ij} = (\sigma uv^*)_{ij} + E_{ij}$ satisfies the conditions (\ref{eq:cond-1}) -- (\ref{eq:cond-4}), then this element definitely belongs to the ``good'' row.
Indeed, from the following chain of equalities and inequalities
\[
\begin{array}{llll}
|a_{ij}| & = & | (\sigma uv^*)_{ij} + E_{ij} | & (\mbox{from~(\ref{eq:cond-2})}) \\
 & = &  | (\sigma uv^*)_{ij}| + |E_{ij} | & (\mbox{from~(\ref{eq:cond-4})}) \\
 & \geqslant & | \sigma |u_i| |v_j| + \left(1 - \varepsilon_0 \right) \delta &\\
 &  = &   \sigma |u_i| \mu_0 \frac{|v_j|}{\mu_0} + \left(1 - \varepsilon_0 \right) \delta &
(\mbox{from~(\ref{eq:cond-3})}) \\
 & \geqslant &   \sigma \mu_1 \|u \|_\infty |v_j| + \frac{|v_j|}{\mu_0} \varepsilon_0 \delta 
 + \left(1 - \varepsilon_0 \right) \delta & (\mbox{from~(\ref{eq:cond-1})}) \\
 & > & \sigma \mu_1 \|u\|_\infty |v_j| + \delta, &
\end{array}
\]
it follows that $|u_i| > \mu_1 \|u\|_\infty$, and the row number $i$ is ``good''.

Secondly, if the column has at least one element that matches the conditions (\ref{eq:cond-1}) -- (\ref{eq:cond-4}), then the maximum in modulus element of this column belongs to the ``good'' row too.
Indeed, suppose that $a$ is the maximum in modulus element in column $j$, and $a_{ij}$ satisfies (\ref{eq:cond-1}) -- (\ref{eq:cond-4}). 
Then
\[
|a| \geqslant |a_{ij}| > \sigma \mu_1 \|u\|_\infty |v_j| + \delta, 
\]
which is equivalent to the fact that $a$ belongs to the ``good'' row. 
At the same time, generally speaking, for the maximum in modulus element in the column all of the conditions need not to be fulfilled.

Finally, we note that (\ref{eq:cond-2}) -- (\ref{eq:cond-3}) define independent events on the set of matrix elements.

Let us fix the index $j$. 
Suppose, that the condition (\ref{eq:cond-1}) holds for a column number $j$  
(that is, $|v_j| \geqslant \mu_0$). 
We estimate the probability that at least one element in this column fulfils the other three conditions (\ref{eq:cond-2}) -- (\ref{eq:cond-4}). 
In view of the above, this assessment will also estimate the probability that the maximum element of $j$-th column belongs to the ``good'' row or equivalently that one step of the $\operatorname{maxvol}$ algorithm gives an element in a ``good'' row.

First of all, we estimate the probability that exactly $k$ elements in the $j$-th column of the matrix $E$ are within $\varepsilon_0 \delta$ of the maximum. For any element this probability is not less then $\varepsilon_0$. So, for a set of $k$ elements we can take
\[
  {P_k} = C_n^k{\left( {\varepsilon_0} \right)^k}{\left( {1 - \varepsilon_0} \right)^{n - k}}.
\]
From the independence of matrix elements in $\sigma u v^*$ and $E$, we get the probability of fulfilling (\ref{eq:cond-2}) equal to $\frac12$.

Let the number of elements that satisfy the condition (\ref{eq:cond-3}) be equal to $l$. Under this condition, the probability (in fact this is a conditional probability) of fulfilment (\ref{eq:cond-3}) is not less than $\frac{l}{n}$
\[
\mathbb{P} \left(\sigma |u_i| \mu_0 \geq \sigma \mu_1 \|u\|_\infty \mu_0 + \varepsilon_0 \delta \right) \geqslant \frac{l}{n}.
\]
Thus, for the considered random realizations the probability of an arbitrary element of $j$-th column to simultaneously satisfy the conditions (\ref{eq:cond-2}) and (\ref{eq:cond-3}) is not less than $\frac{l}{2n}$, and the probability of violation of at least one of them is not more than $1 - \frac{l}{2n}$.

Now we evaluate the probability ${\cal P}_1$ that the column with number $j$ has no elements satisfying all the conditions. 
This assessment can be in obvious way written as
\begin{align}
 {\cal P}_1  & =  \sum\limits_{k = 0}^n {{P_k}{{\left( {1 - \frac{l}{2n}} \right)}^k}} \hfill \nonumber \\  
             & =  \sum\limits_{k = 0}^n {C_n^k{{\left( {1 - \frac{l}{2n}} \right)}^k}} {\left( {\varepsilon_0} \right)^k}{\left( {1 - \varepsilon_0} \right)^{n - k}} \hfill \nonumber \\
   & = {\left( {\varepsilon_0\left( {1 - \frac{l}{2n}} \right) + 1 - \varepsilon_0} \right)^n} \hfill \nonumber \\ 
   & = {\left( {1 - \frac{{l\varepsilon_0}}{2n}} \right)^n}. \hfill \nonumber
\end{align}

After $k$ steps, this probability will not exceed
\[
{\left( {1 - \frac{{l\varepsilon_0}}{2n}} \right)^{nk}} \leqslant {\left( {1 - \frac{\varepsilon_0}{2n}} \right)^{nlk}}.
\]

Now we need to sum this value for all values of $l$, thereby calculating the total probability. Denoting by $\gamma$ the probability of satisfying (\ref{eq:cond-3}) independently for all elements (analogue to $1 - \beta$), we get that
\begin{align}
  {P_2} & = \sum\limits_{l = 0}^n {C_n^l{\gamma ^l}{{\left( {1 - \gamma } \right)}^{n - l}}{{\left( {1 - \frac{{{\varepsilon _0}}}{2n}} \right)}^{nkl}}} \nonumber \\
  & = {\left( {\gamma {{\left( {1 - \frac{{{\varepsilon _0}}}{2n}} \right)}^{kn}} + 1 - \gamma } \right)^n} \nonumber \\
  & \leqslant {\left( {1 - \frac{{{\gamma\varepsilon _0}k}}{2}\left( {1 - \frac{{{\varepsilon _0}k}}{4}} \right)} \right)^n} \nonumber \\
  & \leqslant {e^{-\frac{{{\gamma\varepsilon _0}kn}}{2}\left( {1 - \frac{{{\varepsilon _0}k}}{4}} \right)}}. \nonumber
\end{align} 
The probabilities for rows and columns can be calculated independently, so as a result, since it is still all raised to the power $k$, no matter how many steps has been done in rows and how many in columns.

To complete the proof it remains to estimate the probability ${\cal P}_2$ that at least half of the column elements satisfy (\ref{eq:cond-3}). 
The latter is equivalent to the assertion that half of the components of the uniformly distributed on the sphere vector $u$ satisfy the inequality
\begin{equation} \label{eq:calp_2}
|u_i| \geqslant \mu_1 \|u\|_\infty + \frac{\varepsilon_0 \cdot \delta}{\mu_0 \cdot \sigma}. 
\end{equation}
We use lemma~\ref{l1} to get the estimates on ${\cal P}_2$. 
For this purpose, we define variables $\tau$ and $\gamma$ as
\begin{eqnarray}
\tau  & = & \mu_1 \|u\|_\infty  + \frac{{\varepsilon_0 \delta }}{{\sigma {\mu _0}}}, \nonumber \\
\gamma  & = & 1 - \sqrt {\frac{{2 \tau^2 \left( {n - 2 - 2\sqrt {c\left( {n - 2} \right)\ln n} } \right)}}{\pi }}. \nonumber
\end{eqnarray}
By lemma~\ref{l1} for an arbitrary set of $k$ elements, the probability that all the absolute values are less than $\hat{\mu}$ does not exceed $\gamma^k$.

In addition, we need to take into account the requirement $|v_j| \geqslant \mu_0$. This can be done by simply adding the probability of the opposite. Even if a few steps have been made, this probability cannot reduce: if we have already reached a good column or row, this probability is just zero, so adding the condition that until now such a row or column is not found, does not change the distribution. And the fact that the column was chosen not randomly, but using the algorithm, as shown above, only reduces the probability of the opposite.
As for the fact that some elements might have already been viewed, we can ignore them: if a ``good'' row or column has not been previously found, then they are ``bad'', and discarding them from the consideration only increases the probability of finding a ``good'' one (we evaluate the probability under this particular condition).

Now we can choose $\varepsilon_0$ and $\mu_0$ 
\[
\begin{gathered}
  \varepsilon_0 = \frac{2\ln n}{n} \hfill
  \\
  \mu_0 = c_0 \frac{\ln n}{n \sqrt{\frac{2(n-2 - 2\sqrt{c(n-2)\ln n})}{\pi}}} \hfill
\end{gathered}
\]
Then, firstly,
\[
\mathbb{P}\left( {\left| {{v_j}} \right| \leqslant {\mu _0}} \right) \leqslant {c_0}\frac{{\ln n}}{n},
\]
secondly,
\[
{\cal P}_2 \leqslant 2 \alpha_0 n^{-\gamma k}
\]
and thirdly, we estimate the value of $\gamma$
\[\begin{gathered}
  \gamma  = 1 - \beta  - \frac{{\varepsilon_0 \delta }}{{\sigma {\mu _0}}}\sqrt {\frac{{2\left( {n - 2 + 2\sqrt {c\left( {n - 2} \right)\ln n} } \right)}}{\pi }}  \geqslant  \hfill \\
   \geqslant 1 - \beta  - \frac{2 \varepsilon \|u\|_\infty \|v\|_\infty}{c_0} \cdot \frac{{2\left( {n - 2 + 2\sqrt {c\left( {n - 2} \right)\ln n} } \right)}}{{\pi}}. \hfill \\ 
\end{gathered} \]
Putting together all the probabilities, we find that the probability to get to the ``good'' element after $k$ 
steps of the algorithm is bounded from above by
\[
1 - 2\alpha n^{-c} - \alpha_0 n^{-\gamma k} - \left( \frac{c_0 \ln n}{n} \right)^k.
\]
\end{proof}

It is easy to see that  for the probability of order $n^{-k}$ the number of steps does not depend on $n$.

For simplicity we have taken a square matrix, but the claim is easily generalized to the case $ m \neq n $: for this case, instead of just multiplying by 2 we can take separate terms for $ u $ and $ v $. 
The steps  for the rows and columns also should be considered separately due to their different sizes.

It can also be generalized to the complex case: it is enough to take the value of $\varepsilon_0$a little more and replace (\ref{eq:cond-2}) by the condition on the smallness of the phase.

\section{Numerical experiments}\label{calc-sec}

Before proceeding to the calculations, it is important to understand what happens if the inequality (\ref{ed}) is not satisfied. 
In this case, the error will be about $C$-norm of the whole matrix. 
In the worst case, it is
\[
|d - ba^{-1}c| \leqslant |d| + |a| \leqslant \sigma  \left\| u \right\|_{\infty} \left\| v \right\|_{\infty} + \delta + |a|.
\]
If $|a|$ is sufficiently large, then, as we already know, the error can not exceed $\frac{|d|}{|a|}4 \delta$.

Taking the minimum of $4 \delta \frac{|d|}{|a|}$ and $|d| + |a|$ and substituting the estimate for $d$, we find that the error will not exceed
\begin{equation}\label{maxd}
\frac{1 + \delta + \sqrt{(1 + \delta)(1 + 17\delta)}}{2}.
\end{equation}

\begin{figure}[ht]
\centering
\includegraphics[width=\columnwidth]{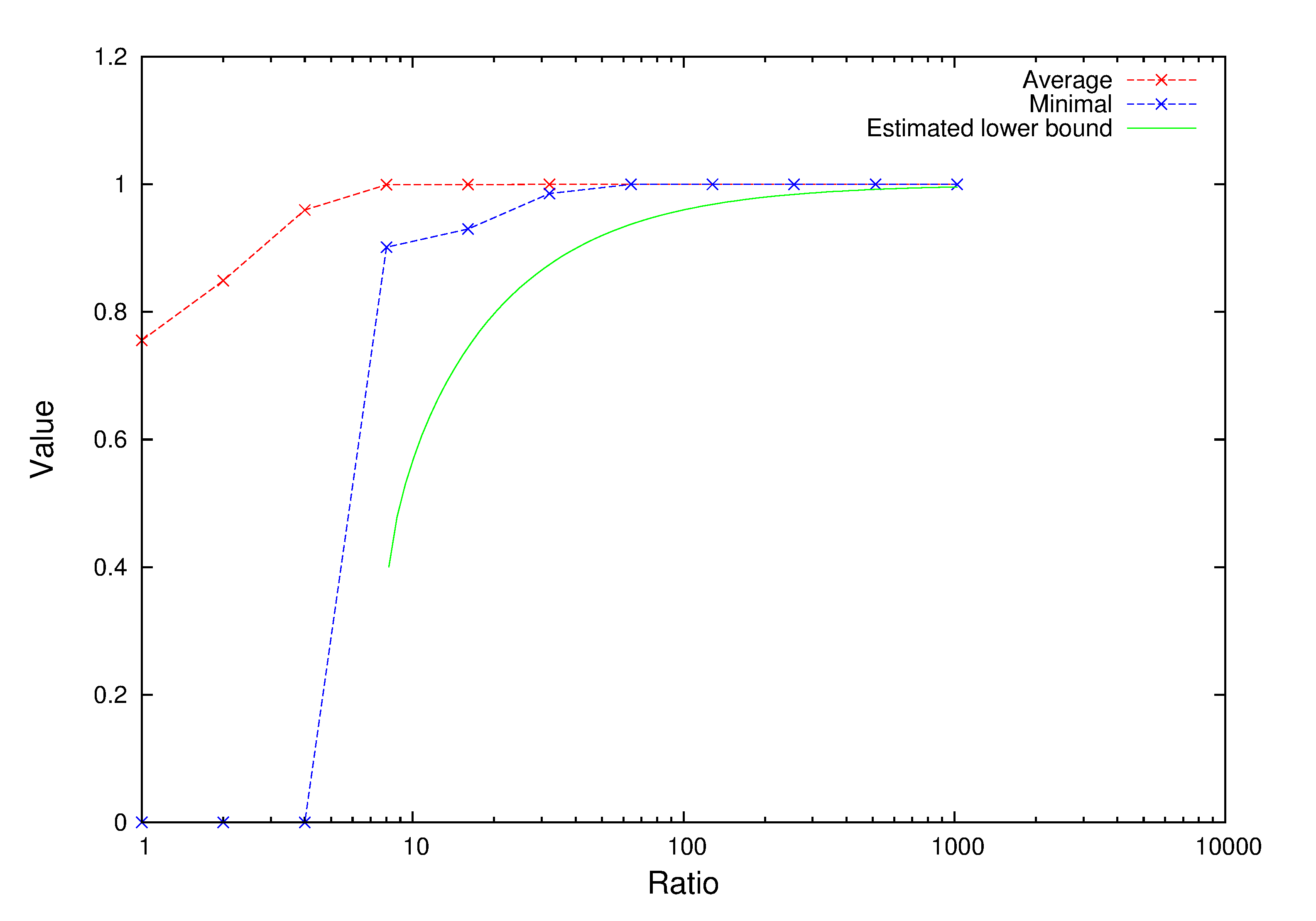}
\caption{The dependence of the Value = ratio between the found element and the maximum from the Ratio = $\frac{\sigma  \left\| u \right\|_{\infty} \left\| v \right\|_{\infty}}{\delta}$.
We show the mean, and the minimum for $1000$ matrix generations, and the lower bound estimate which is equal to $\sigma \mu_2^2  \left\| u \right\|_{\infty} \left\| v \right\|_{\infty} + \delta$.}\label{pic1}
\end{figure}

To verify the accuracy of the estimates, the calculations were carried out for the random matrices. Namely, the matrix was set by its singular value decomposition. 
The left and right singular vectors were randomly selected, the first singular value was selected from the equation
\[
x = \frac{\sigma  \left\| u \right\|_{\infty} \left\| v \right\|_{\infty}}{\delta},
\]
and all the rest singular values were set to 1.

The value of $x$ was placed on the horizontal axis. 
If $x \geqslant 8$, then we verified that the column is ``good'' before applying $\operatorname{maxvol}$. 
Figure \ref{pic1} illustrates the relationship between the found element and the maximum element of the matrix. 
Figure \ref{pic2} shows the approximation error. Figure \ref{pic3} shows the probability of hitting a ``bad'' column.

\begin{figure}[ht]
\centering
\includegraphics[width=\columnwidth]{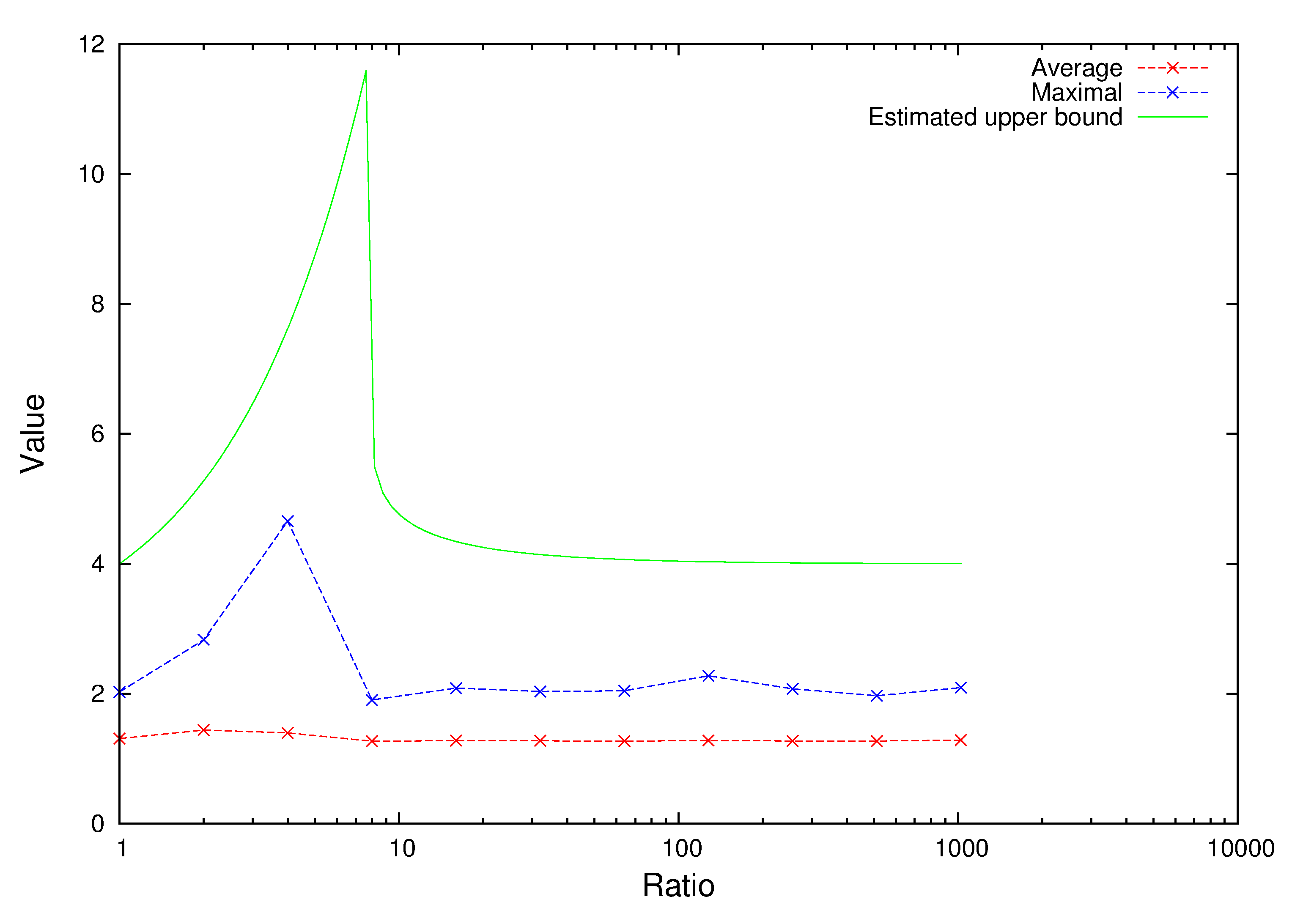}
\caption{The dependence of the Value = ratio between approximation error with respect to $\delta$ from the Ratio = $\frac{\sigma  \left\| u \right\|_{\infty} \left\| v \right\|_{\infty}}{\delta}$. 
We show the mean, and the minimum for $1000$ matrix generations, and the estimate of the error. 
If $\frac{1}{\varepsilon} < 8$ then the expression (\ref{maxd}) was used.}\label{pic2}
\end{figure}

\begin{figure}[ht]
\centering
\includegraphics[width=\columnwidth]{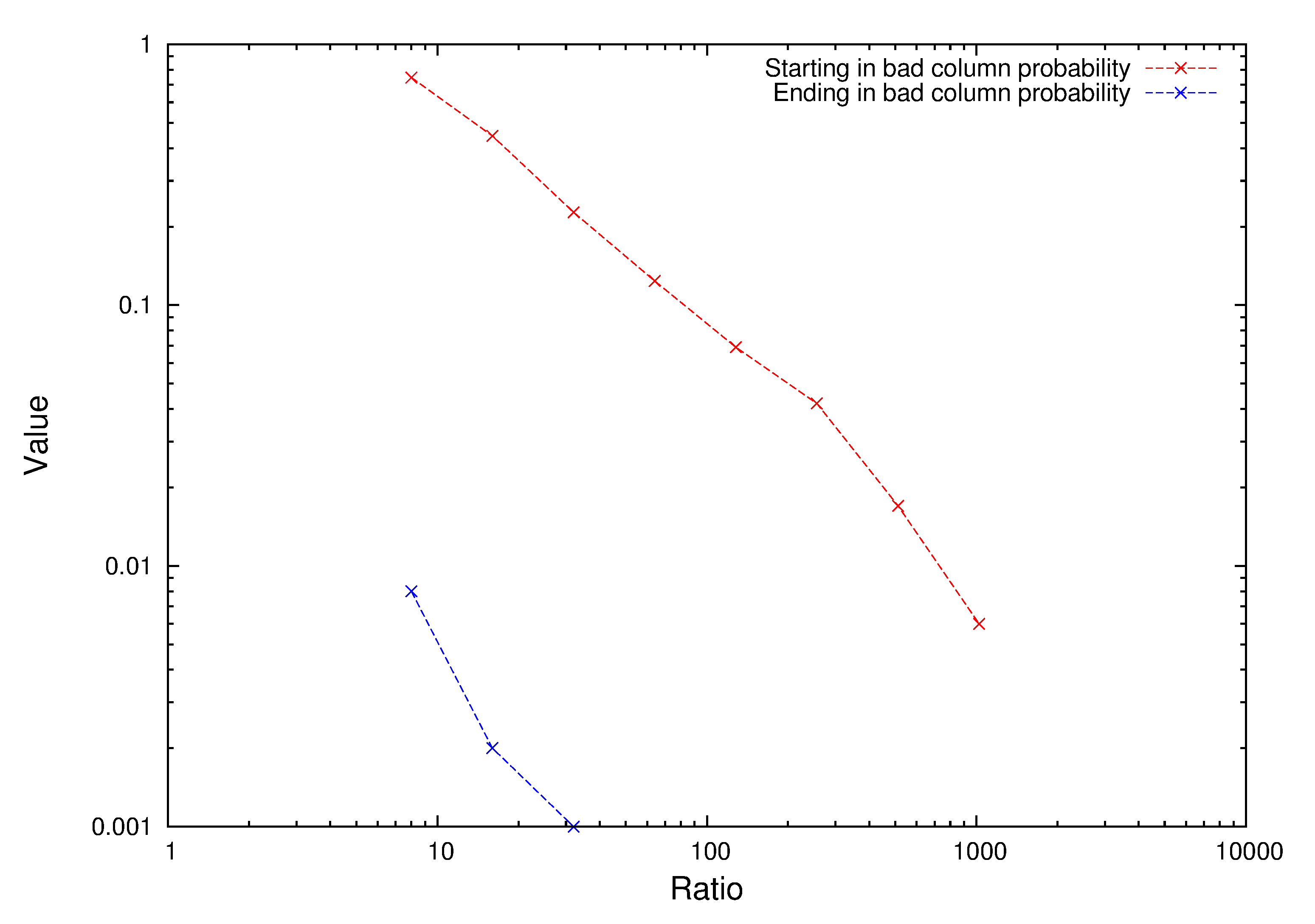}
\caption{The dependence of the Value = probability to get into the ``bad" column from the Ratio $\frac{\sigma  \left\| u \right\|_{\infty} \left\| v \right\|_{\infty}}{\delta}$. 
The results are obtained for the $1000$ matrix generations.
The probabilities are shown for randomly selected columns and the columns obtained after applying the algorithm.
}\label{pic3}
\end{figure}

It is seen that the maximum error is different from our estimate about 2 times. This is probably due to the fact that the selected matrix $ E $ is the best approximation in $ 2 $-norm, but not in $ C $-norm.

The last figure shows that the probability of not getting into the ``good'' column almost vanishes after the application of the algorithm. This means that the matrix of the best rank 1 approximation and its error are not closely related, and even if we start from a ``bad'' column there is a great chance to get eventually to a ``good'' one.

\section{Conclusion}

We proved that 
in the important particular case of rank one approximations
algorithm $\operatorname{maxvol}$ applied for the random matrices 
finds the element close to the maximum with the probability close to 1.
This guarantees the high accuracy of the cross approximations.

\section{Acknowledgments}
The work was supported by the Russian Science Foundation, Grant 14-11-00806.

\newpage
\par\bigskip\noindent
{\hfil\bf Bibliography\hfil}\par
\end{document}